\documentclass[11pt, a4paper, reqno]{article}
\usepackage[utf8]{inputenc}
\usepackage[T1]{fontenc}
\usepackage{amsmath}
\usepackage{mathtools}
\usepackage{amsthm}
\usepackage{stmaryrd}
\usepackage[normalem]{ulem}
\IfFileExists{dsfont.sty}{\usepackage{dsfont}}{\providecommand{\mathds}[1]{\mathbf{##1}}}
\IfFileExists{newtxtext.sty}{%
  \usepackage{newtxtext}
  \usepackage[cmintegrals,cmbraces]{newtxmath}
  \usepackage[cal=boondoxo]{mathalfa}
}{%
}
\usepackage[a4paper, margin=1.2in, top=1.2in, bottom=1.3in]{geometry}
\IfFileExists{mleftright.sty}{\usepackage{mleftright}\mleftright}{}
\allowdisplaybreaks
\usepackage{enumitem}
\setlist{itemsep=0.3em, topsep=0.4em, parsep=0.1em}
\setlist[enumerate]{label=\normalfont(\roman*)}
\usepackage{tikz}
\usetikzlibrary{calc, arrows.meta}
\usepackage{xcolor}
\usepackage[font=small, labelfont=sc, labelsep=period]{caption}
\usepackage{titlesec}
\titleformat{\section}[block]{\large\bfseries\centering}{\thesection.}{0.8em}{}
\titleformat{\subsection}[hang]{\normalsize\bfseries}{\thesubsection.}{0.8em}{}
\definecolor{marine}{RGB}{0, 32, 128}
\definecolor{crimson}{RGB}{220, 20, 60}
\definecolor{graybeam}{gray}{0.75}
\usepackage[colorlinks=true, linkcolor=marine, citecolor=marine, urlcolor=marine]{hyperref}

\DeclareMathOperator{\rank}{rank}
\newcommand{\R}{\mathbb R}
\newcommand{\X}{\R^{m\times d}}
\newcommand{\Ob}{\overline{\Omega}}
\newcommand{\Meas}{\mathcal M}
\newcommand{\Measp}{\mathcal M^{+}}
\newcommand{\Prob}{\mathrm{Prob}}
\newcommand{\Leb}{\mathcal L}
\newcommand{\Haus}{\mathcal H}
\newcommand{\restr}{\mathbin{\vrule height 1.6ex depth 0pt width 0.13ex\vrule height 0.13ex depth 0pt width 1.3ex}}
\newcommand{\YM}{\mathbf Y}
\newcommand{\bnu}{\boldsymbol\nu}
\newcommand{\toY}{\stackrel{\YM}{\to}}
\newcommand{\toweakstar}{\stackrel{*}{\rightharpoonup}}
\newcommand{\bTheta}{\boldsymbol\Theta}
\newcommand{\Div}{\operatorname{div}}
\newcommand{\vol}{\mathcal L^d}
\newcommand{\fl}[1]{\lVert #1\rVert_{\mathrm{KR}}^\Omega}
\newcommand{\la}{\langle}
\newcommand{\ra}{\rangle}
\newcommand{\Lip}{\mathrm{Lip}}
\newcommand{\spt}{\operatorname{spt}}
\newcommand{\Span}{\operatorname{span}}
\newcommand{\dist}{\operatorname{dist}}

\numberwithin{equation}{section}
\newtheoremstyle{eleganttheorem}
  {0.8em plus 0.2em minus 0.2em}{0.8em plus 0.2em minus 0.2em}
  {\itshape}{}{\bfseries}{.}{0.5em}{}
\newtheoremstyle{elegantremark}
  {0.8em plus 0.2em minus 0.2em}{0.8em plus 0.2em minus 0.2em}
  {\normalfont}{}{\itshape}{.}{0.5em}{}
\theoremstyle{eleganttheorem}
\newtheorem{theorem}{Theorem}[section]
\newtheorem{proposition}[theorem]{Proposition}
\newtheorem{lemma}[theorem]{Lemma}
\newtheorem{corollary}[theorem]{Corollary}
\newtheorem{definition}[theorem]{Definition}
\theoremstyle{elegantremark}
\newtheorem{remark}[theorem]{Remark}

\title{{\textbf{\Large Divergence-free concentrations come from vanishing sequences}}}
\author{{Adolfo Arroyo-Rabasa$^{\,\dagger}$ \qquad Francesco Nobili$^{\,*}$ \qquad Ivan Yuri Violo$^{\,\dagger}$}\\[0.6em]
  {\normalsize\itshape $^{\dagger}$Dipartimento di Matematica, Università di Pisa}\\
  {\normalsize\itshape $^{*}$Dipartimento di
Matematica e Applicazioni, Università di Napoli Federico II}}
\date{}

\begin{document}
\maketitle

\begin{abstract}
\noindent
A vanishing sequence $V_n$ of divergence-free matrix fields is one that is bounded in $L^1$ and is carried
by open sets $A_n$ of vanishing volume. Recently established,  Bouchitté's vanishing mass conjecture says that the directions such a sequence can carry are rigidly constrained: their limiting
distribution must be a superposition of microstructures whose barycenters are singular matrices. We
prove the converse in the non-symmetric setting: every such superposition is attained by a vanishing sequence. In fact, we are able to construct divergence-free fields $V_n$ \emph{supported} on $A_n$, whose relative boundary is a smooth compact manifold.
\begin{flushleft}
    
\end{flushleft}
\end{abstract}

%=====================================================================
\section{Introduction}\label{sec:intro}
%=====================================================================

Let $m \ge d\ge2$ be integers. 
We fix $\Omega\subset\R^d$ a bounded connected open set, and consider matrix
fields $V = (v^1,\dots,v^m) \in L^1(\Omega;\X)$ subject to the row-wise divergence constraint
\[
    \Div V \ = \ \bigl(\Div v^1,\dots,\Div v^m\bigr) \ = \ 0
    \qquad\text{in }\mathcal D'(\Omega;\R^m) .
\]
We call any such $V$ a divergence-free field. We write $|\cdot|$ for the Euclidean norm on $\X$ and $B$ for its corresponding open unit ball. We also write $\partial_\Omega A\coloneqq \Omega\cap \partial A$ for the topological boundary of a set $A\subset \Omega$ \emph{relative} to $\Omega$.

\begin{definition}[Vanishing sequence]\label{def:vanishing}
A sequence $(V_n)_n\subset L^1(\Omega;\X)$ is \emph{vanishing} if there are open sets
$A_n\subset\Omega$ with
\begin{equation}\label{eq:vanishing}
    \|V_n\|_{L^1(\Omega)} = 1, \qquad  \int_{\Omega \setminus A_n} |V_n| \, dx = 0, \qquad \mathcal L^d(A_n)\to0 .
\end{equation}
We call the sequence \emph{regular vanishing} if moreover 
\[
\text{
$\partial \Omega \subset \partial A_n$, \qquad $\partial_\Omega A_n$ \ is a compact smooth manifold, \qquad $\spt V_n \subset A_n$.
}
\]
\end{definition}
Such a sequence carries no limit in the usual sense: it converges to zero in measure while keeping
some mass. Following \cite{Alberti21,GennaioliRindler26}, the directions it does carry are recorded by the
push-forward
\begin{equation}\label{eq:theta}
    \Theta_V \ \coloneqq\ \Bigl(\tfrac{V}{|V|}\Bigr)_{\#}\bigl(\mathds 1_{\{|V|>0\}}|V|\,\Leb^d\bigr)
    \ \in\ \Measp(\partial B) ,
\end{equation}
so that $\Theta_V(E)$ is the amount of $L^1$-mass that $V$ carries in the directions
$E\subset\partial B$, and $\Theta_V(\partial B) = \|V\|_{L^1(\Omega)}$. Since our task is to construct sequences, we shall also keep track of \emph{where} the directions occur, through the localized version
\begin{equation}\label{eq:btheta}
    \bTheta_V \ \coloneqq\ \Bigl(\mathrm{id},\tfrac{V}{|V|}\Bigr)_{\#}\bigl(\mathds 1_{\{|V|>0\}}|V|\,\Leb^d\bigr)
    \ \in\ \Measp(\Ob\times\partial B) ,
\end{equation}
whose first marginal is $|V|\Leb^d$ and whose second marginal is precisely $\Theta_V$. Along a vanishing sequence
the total masses are bounded, so, up to a subsequence,
\begin{equation}\label{eq:limitpair}
\begin{gathered}
    \bTheta_{V_n}\ \toweakstar\ \lambda\otimes\mu_x \quad\text{in }\Meas(\Ob\times\partial B),
    \qquad\text{and hence}\\[0.3em]
    \Theta_{V_n}\ \toweakstar\ \mu \ \coloneqq\ \int_{\Ob}\mu_x\,d\lambda(x) \ \in\ \Prob(\partial B)
    \qquad\text{in }\Meas(\partial B) ,
\end{gathered}
\end{equation}
where $\lambda\in\Prob(\Ob)$ is the weak-$*$ limit of $|V_n|\Leb^d$ and $x\mapsto\mu_x\in\Prob(\partial B)$ is the $\lambda$-measurable family obtained by disintegrating the limit of $\bTheta_{V_n}$ with respect to its first marginal. We call the pair $(\lambda,\mu_x)$ the \emph{limiting concentration} of the sequence, and $\mu$ its \emph{averaged concentration}: the former records where, and along which directions, mass concentrates; the latter records the directions along
which mass concentrates, averaged over $\Omega$. Write
\begin{equation}\label{eq:cone}
\begin{aligned}
   \Lambda = \Lambda_{m\times d} &\coloneqq \{ M \in \X : \rank M < d \} \\
            & \ = \ \bigl\{ M \in \X : \det(M_I) = 0 \ \text{ for every } d \times d \text{ minor } M_I \bigr\}
\end{aligned}
\end{equation}
for the cone of \emph{rank-deficient} matrices, so that $\Lambda_{d\times d}$ is the cone of
singular $d\times d$ matrices.  The cone $\Lambda$  is closed and symmetric and spans $\X$, since it contains every rank-one matrix. In its original form (see Sect.~\ref{sec:symmetric}), Bouchitté's
vanishing mass conjecture \cite{Bouchitte01}, which originates in the optimal design of light
structures \cite{Michell04,BouchitteGangboSeppecher08,BabadjianIurlanoRindler23,Gangbo18}, concerns
the concentrations generated by divergence-free vanishing sequences. Gennaioli and Rindler have
recently established it, in the stronger form in which the sequence is only required to converge to
zero in measure:

\begin{theorem}[Gennaioli--Rindler \cite{GennaioliRindler26}]\label{thm:VMC}
Let $(V_n)_n\subset L^1(\Omega;\X)$ be uniformly bounded, divergence-free and converging to zero in
measure, with averaged concentration $\mu\in\Prob(\partial B)$. Then there are a probability measure
$\pi\in\Prob(\Lambda)$ and a weakly Borel family $\{\nu_X\}_{X\in\Lambda}\subset\Prob(\partial B)$ (defined
for $\pi$ a.e. $X$) with
\begin{equation}\label{eq:superposition}
    \mu \ = \ \int_{\Lambda}\nu_X\,d\pi(X)
    \qquad\text{and}\qquad
    \bar\nu_X \ = \ X \quad \pi\text{-a.e.,}
\end{equation}
where $\bar\nu$ denotes the barycenter of $\nu \in \Prob(\partial B) $.
\end{theorem}

Call a measure $\mu\in\Prob(\partial B)$ satisfying \eqref{eq:superposition} a
\emph{$\Lambda$-superposition}: an average of \emph{microstructures} $\nu_X$, each of which has a
rank-deficient barycenter. Theorem~\ref{thm:VMC} says that nothing but a $\Lambda$-superposition can
occur, on average. The statement localizes as follows.

\begin{remark}[Localization]\label{rem:localization}
Let $(V_n)_n$ be as in Theorem~\ref{thm:VMC}, with limiting concentration $(\lambda,\mu_x)$. Then:
\begin{enumerate}
    \item[(i)] $\Div(\bar \mu_x\,\lambda(dx)) = 0$ in $\mathcal D'(\Omega;\R^m)$,
    \item[(ii)] $\mu_x$ is a $\Lambda$-superposition for $\lambda$-a.e.\ $x\in\Omega$.
\end{enumerate}
Condition (i) is elementary:
\[
    \int_\Omega\varphi\,V_n\,dx=\int\varphi(x)\,z\,d\bTheta_{V_n}(x,z)\to\int_{\Ob}\varphi\,\bar \mu_x\,d\lambda,\qquad \text{for all }\varphi\in C(\Ob),
    \]
    and the constraint $\Div V_n=0$ passes to the limit. Condition (ii) is the localized form of Theorem~\ref{thm:VMC}: it follows from that theorem and from the localization principles for divergence-free Young measures of \cite{ARDPR_Advances} (see Lemma~\ref{lem:localization}).
\end{remark}

Conditions (i)--(ii) are thus necessary for the limiting concentration of a divergence-free sequence converging to zero in measure. The converse of Theorem~\ref{thm:VMC}, which says that every $\Lambda$-superposition can be attained by a divergence-free sequence converging to zero in measure, follows from the convexity results in \cite{KirchheimKristensen16} and the characterization of $\mathcal A$-free Young measures from \cite{ArroyoRabasa20}. Our result shows that, in addition, the recovery sequence can be replaced by a regular vanishing one, and that the same holds in the localized form: conditions (i)--(ii) are also sufficient.

\begin{theorem}\label{thm:main}
Let $\Omega\subset\R^d$ be a bounded connected open set, let $\lambda\in\Prob(\Ob)$ satisfy $\lambda(\partial\Omega)=0$, and let $x\mapsto\mu_x\in\Prob(\partial B)$ be a (weakly) $\lambda$-measurable family satisfying (i) and (ii) of Remark~\ref{rem:localization}. Then there is a regular vanishing sequence $(V_n)_n\subset L^1(\Omega;\X)$ of
divergence-free fields with limiting concentration $(\lambda,\mu_x)$:
\[
    \bTheta_{V_n}\ \toweakstar\ \lambda\otimes\mu_x
    \qquad\text{in }\Meas(\Ob\times\partial B) .
\]
In particular $\Theta_{V_n}\toweakstar\mu=\int_{\Ob}\mu_x\,d\lambda$ in $\Meas(\partial B)$.
\end{theorem}

\begin{corollary}\label{cor:homogeneous}
Let $\Omega\subset\R^d$ be a bounded connected open set and let $\mu\in\Prob(\partial B)$ be a $\Lambda$-superposition. Then there is a regular vanishing sequence $(V_n)_n\subset L^1(\Omega;\X)$ of
divergence-free fields with
\[
    \Theta_{V_n}\ \toweakstar\ \mu
    \qquad\text{in }\Meas(\partial B).
\]
\end{corollary}
\begin{proof}
Apply Theorem~\ref{thm:main} with $\lambda\coloneqq\vol(\Omega)^{-1}\Leb^d\restr\Omega$ and $\mu_x\coloneqq\mu$ for every $x$: then $\lambda(\partial\Omega)=0$, (ii) holds by assumption, and (i) holds because $\bar \mu_x\lambda=\vol(\Omega)^{-1}\bar\mu\,\Leb^d\restr\Omega$ is a constant field on $\Omega$, hence divergence-free.
\end{proof}
\medskip
\noindent\textbf{The filling lemma.} The whole difficulty is in the word \emph{vanishing}. Out of (i)--(ii), the theory of $\mathcal A$-free generalized Young
measures~\cite{ArroyoRabasa20} produces a divergence-free $C^1$ sequence converging to zero in measure
(see \S\ref{sec:reduction}); but such a sequence needs not be supported on small sets, and truncating
it may destroy the PDE constraint. What has to be repaired is the \emph{truncated field}: the part of $V$
living on a small set $A$ is no longer divergence-free, and
$\Div(V\mathds 1_A) = -\Div(V-V\mathds 1_A)$ is the \emph{load} it sheds. The repair is the following
lemma, which we regard as the main technical contribution of this work.

\begin{lemma}[Filling by currents]\label{lem:filling}
Let $\Omega\subset\R^d$ be bounded, connected and open, let $v\in L^1(\Omega;\R^d)$ have compact support in
$\Omega$, and let $f\coloneqq\Div v$. Then for every $\eta>0$ there is ${w}\in L^1(\Omega;\R^d)$ with
\[
    \Div {w} = f \ \text{ in }\mathcal D'(\Omega),
    \qquad \int_\Omega|{w}|\,dx \ \le\ \fl{f}+\eta,
    \qquad \vol( \{ w\neq 0\}) \ \le\ \eta ,
\]
where $\fl\cdot$ is the Kantorovich--Rubinstein norm of the load relative to $\Omega$, the supremum
of $u \mapsto \la f,u\ra$ over the functions that are $1$-Lipschitz for the geodesic distance of $\Omega$ (see \eqref{eq:KR}).

If moreover $f$ is a measure supported on a compact Lebesgue null-set $K\subset \Omega$, then we have in addition that
\[
\vol( \spt w) \ \le\ \eta.
\]
\end{lemma}
A load can therefore always be equilibrated by a field that is both \emph{light} of almost
optimal cost and \emph{thin} (supported on a set of arbitrarily small volume). Our proof is
geometric and rests on Arens--Eells duality through the ideas introduced in~\cite{ArroyoRabasaBouchitte25}. It matters that the cost is measured by the
\emph{intrinsic} Kantorovich--Rubinstein norm rather than by Whitney's flat norm \cite{Whitney57}:
the two agree when the geodesic distance of $\Omega$ is Euclidean and are comparable exactly when
$\Omega$ is quasiconvex (in the metric sense)~\cite{ArroyoRabasaBouchitte25,bate2025structure}, but they differ in general, and only the intrinsic one
is finite exactly for the loads that can be balanced \emph{inside} $\Omega$. This is what makes
Lemma~\ref{lem:filling} sharp, regardless of the domain's geometry.

We point out that replacing $\{w \neq 0\}$ with $\spt w$ will be crucial for our goals, especially to exhibit a \emph{regular} vanishing sequence. Since $\spt f\subset\spt w$ whenever $\Div w=f$, this improvement cannot be expected if $v\in L^1$ only: the support of the load $f=\Div v$ may then have positive measure (take $f=\phi\,\Leb^d$ with $\phi\in C^\infty_c(\Omega)$ of zero mean and $\{\phi\neq0\}$ of positive measure), in which case every filling $w$ of $f$ satisfies $\vol(\spt w)\ge\vol(\spt f)$.
 In fact, we will exploit the existence of $C^1$ recovery sequences to obtain a refined duality result for measures of this kind (cf.\ Lemma \ref{lem:refine}) that we believe of interest in its own right.

\medskip
\noindent\textbf{Declaration of generative AI use:} During the preparation of this work, the authors used AI solely to assist in generating the illustration. All mathematical results, proofs, and analyses were conducted independently by the authors, who take full responsibility for the contents of this note.

%=====================================================================
\section{Notation} 
%=====================================================================
$\Meas(\Omega;\X)$ denotes the finite $\X$-valued Radon measures,
$\Measp$ the positive ones and $\Prob$ the probability measures; $\Leb^d$, $\Haus^k$ are the Lebesgue
and Hausdorff measures; $\bar\nu\coloneqq\int z\,d\nu(z)$ is the barycenter of $\nu$, and
$\llbracket a,b\rrbracket$ the oriented segment from $a$ to $b$. The length $\int_0^1 |\gamma'| \, dt$ of a Lipschitz curve $\gamma: [0,1] \to \Omega$ will be denoted by $\ell(\gamma)$. Notice that $\Haus^1(\gamma([0,1]))\le \ell(\gamma)$, with equality when $\gamma$ is injective.

%=====================================================================
\section{Reduction to a modification lemma}\label{sec:reduction}
%=====================================================================

Fix $\lambda$ and $\{\mu_x\}$ as in Theorem~\ref{thm:main}. The object we shall prescribe, simply as an auxiliary tool, is the
generalized Young measure with trivial oscillation part
\begin{equation}\label{eq:choice}
    \bnu \ \coloneqq\ \bigl(\delta_0,\ \lambda,\ \mu_x\bigr) .
\end{equation}
We show
that $\bnu$ satisfies the axioms of divergence-free Young measures under~(i)--(ii), so that \cite{ArroyoRabasa20} produces a divergence-free sequence
generating it: such a sequence converges to zero in measure because the oscillation part is
$\delta_0$, and its limiting concentration is $(\lambda,\mu_x)$ because the concentration part of $\bnu$ is
$(\lambda,\mu_x)$. Only $\vol(\spt V_n)\to0$ is then missing; this is the point of \S\ref{sec:filling}.

\medskip
\noindent\textbf{Generalized Young measures.} A generalized Young measure
\cite{AlibertBouchitte97,DiPernaMajda87} is a triple
$\bnu = (\nu_x,\lambda,\nu^\infty_x)$ in which $x\mapsto\nu_x\in\Prob(\X)$ is a $\Leb^d$-measurable family with $\int_\Omega\int |A|\,d\nu_x(A)\,dx<\infty$, $\lambda\in\Measp(\Ob)$, and $x\mapsto\nu^\infty_x\in\Prob(\partial B)$ is a $\lambda$-measurable family. A continuous $g:\X\to\R$ is an \emph{admissible integrand} if $|g(A)|\le C(1+|A|)$ and the recession
function $g^\infty(A)\coloneqq\lim_{t\to\infty}t^{-1}g(tA)$ exists and is continuous; $g^{\#}$
denotes the corresponding upper limit.
We say that a sequence
$(V_n)_n\subset L^1(\Omega;\X)$ \emph{generates} a Young measure $\bnu$, writing $V_n\toY\bnu$, if
\begin{equation}\label{eq:generation}
    \int_\Omega g(V_n)\,\varphi\,dx \ \longrightarrow\
    \int_\Omega\la\nu_x,g\ra\,\varphi\,dx+\int_{\Ob}\la\nu^\infty_x,g^\infty\ra\,\varphi\,d\lambda(x)
\end{equation}
for every admissible integrand $g$ and every $\varphi\in C(\Ob)$. Here $\nu_x$ records the oscillation at $x$, $\lambda$ the spatial distribution of the concentrating mass $|V_n|\Leb^d$, and $\nu^\infty_x$ its directions at $x$. Every sequence bounded in $L^1(\Omega;\X)$ has a subsequence generating a Young measure \cite{AlibertBouchitte97,KristensenRindler10}, and testing against the products $\varphi\otimes g$ is equivalent to testing against all integrands $f(x,A)$ of the class considered in \cite{ArroyoRabasa20}, by density and the boundedness of the masses. The barycenter of $\bnu$ is the measure
\[
\bar \bnu \ \coloneqq \ \bar \nu_x \, \Leb^d \restr \Omega \ + \  \bar\nu^\infty_x \, \lambda \ \in \ \Meas(\Ob;\X).
\]
 Only trivial oscillation parts $\nu_x=\delta_0$ will occur here.

\begin{lemma}\label{lem:theta}
Let $(V_n)_n\subset L^1(\Omega;\X)$ be bounded in $L^1$, let $\lambda\in\Measp(\Ob)$ and let $x\mapsto\nu^\infty_x\in\Prob(\partial B)$ be a (weakly) $\lambda$-measurable family. The following are equivalent:
\begin{enumerate}
    \item[(a)] $(V_n)_n$ generates $(\delta_0,\lambda,\nu^\infty_x)$,
    \item[(b)] $V_n\to0$ in measure and
    \[
        \bTheta_{V_n}\ \toweakstar\ \lambda\otimes\nu^\infty_x \quad \text{ in $\Meas(\Ob\times\partial B)$.}
    \]
\end{enumerate}
In this case $\|V_n\|_{L^1(\Omega)}\to\lambda(\Ob)$ and $\Theta_{V_n}\toweakstar\int_{\Ob}\nu^\infty_x\,d\lambda$ in $\Meas(\partial B)$. For the choice \eqref{eq:choice} the limiting concentration is $(\lambda,\mu_x)$.
\end{lemma}

\begin{proof}
(a)$\Rightarrow$(b). Let $\varphi\in C(\Ob)$, $\psi\in C(\partial B)$ and set $g(A)\coloneqq|A|\,\psi(A/|A|)$ for $A\ne0$, $g(0)\coloneqq0$.
Then $g$ is continuous and positively $1$-homogeneous, hence admissible with $g^\infty = g$, and
$g = \psi$ on $\partial B$; by \eqref{eq:generation},
\[
    \la\bTheta_{V_n},\varphi\otimes\psi\ra \ = \ \int_\Omega \varphi\,g(V_n)\,dx
    \ \longrightarrow\ g(0)\int_\Omega\varphi\,dx+\int_{\Ob}\la\nu^\infty_x,\psi\ra\,\varphi\,d\lambda
    \ = \ \la\lambda\otimes\nu^\infty_x,\varphi\otimes\psi\ra .
\]
Taking $\varphi\otimes\psi\equiv1$ shows that the masses $\bTheta_{V_n}(\Ob\times\partial B)=\|V_n\|_{L^1(\Omega)}$ converge to $\lambda(\Ob)$, and since the linear span of the products $\varphi\otimes\psi$ is dense in $C(\Ob\times\partial B)$, we obtain $\bTheta_{V_n}\toweakstar\lambda\otimes\nu^\infty_x$. The statement for $\Theta_{V_n}$ follows by taking $\varphi\equiv1$.
For the convergence in measure, take any $\varepsilon >0$, $\varphi\equiv1$ and $g(A) = (|A| / \varepsilon )\wedge1$, admissible with
$g^\infty\equiv0$. Then, we observe that $\int_\Omega (|V_n|/ \varepsilon) \wedge1\,dx\to\la\delta_0,g\ra\,\vol(\Omega) = 0$ as well as that $\mathds 1_{\{|V_n|>\varepsilon\}} \le  (|V_n| / \varepsilon)\wedge1$ holds a.e.. In particular, we obtain
\[
    \vol(\{|V_n|>\varepsilon\}) \le \int_\Omega\frac{|V_n|}{\varepsilon}\wedge1\,dx\to 0,\qquad \text{as }n\to\infty,
\]
as desired.

(b)$\Rightarrow$(a). Let $(V_{n_j})_j$ be an arbitrary subsequence. Since it is bounded in $L^1$, a further subsequence (not relabelled) generates some Young measure $\tilde\bnu=(\tilde\nu_x,\tilde\lambda,\tilde\nu^\infty_x)$. For $f\in C_c(\X)$, which is admissible with $f^\infty\equiv0$, and for every $\delta>0$ we have $\int_\Omega|f(V_{n_j})-f(0)|\,dx\le2\|f\|_\infty\vol(\{|V_{n_j}|>\delta\})+\omega_f(\delta)\vol(\Omega)$, where $\omega_f$ is the modulus of continuity of $f$. Therefore, taking first $j\to \infty$ and then $\delta\to 0$, we obtain $f(V_{n_j})\to f(0)$ in $L^1(\Omega)$, and testing \eqref{eq:generation} with $f$ gives $\la\tilde\nu_x,f\ra=f(0)$ for $\Leb^d$-a.e.\ $x$. Applying this to a countable dense subset of $C_c(\X)$ we get $\tilde\nu_x=\delta_0$ for a.e.\ $x$. By the implication (a)$\Rightarrow$(b) already proved, $\bTheta_{V_{n_j}}\toweakstar\tilde\lambda\otimes\tilde\nu^\infty_x$, so that $\tilde\lambda\otimes\tilde\nu^\infty_x=\lambda\otimes\nu^\infty_x$ by uniqueness of weak-$*$ limits. Taking first marginals gives $\tilde\lambda=\lambda$, and the uniqueness of the disintegration gives $\tilde\nu^\infty_x=\nu^\infty_x$ for $\lambda$-a.e.\ $x$. Every subsequence thus has a further subsequence generating $(\delta_0,\lambda,\nu^\infty_x)$. Finally, since bounded sets of Young measures are metrizable, the whole sequence does.
\end{proof}

\begin{lemma}[Localization]\label{lem:localization}
Let $(V_n)_n\subset L^1(\Omega;\X)$ be divergence-free, bounded in $L^1$ and converging to zero in measure, with limiting concentration $(\lambda,\mu_x)$. Then $\mu_x$ is a $\Lambda$-superposition for $\lambda$-a.e.\ $x\in\Omega$.
\end{lemma}
\begin{proof}
By Lemma~\ref{lem:theta} (and its last conclusion), $(V_n)_n$ generates $\bnu\coloneqq(\delta_0,\lambda, \mu_x)$, which is therefore a divergence-free generalized Young measure. We use the localization principles for such Young measures \cite{ArroyoRabasa20,ARDPR_Advances}: for $\Leb^d$-a.e.\ $x_0\in\Omega$ (regular points) and for $\lambda^s$-a.e.\ $x_0\in\Omega$ (singular points) there is a tangent Young measure $\sigma$ on the unit cube $Q$, generated by divergence-free fields $W_j\in L^1(Q;\X)$ obtained by rescaling, translating, and renormalizing the restrictions of a subsequence of $(V_n)_n$ to small cubes centred at $x_0$, of the form
\begin{enumerate}
    \item $\sigma=\bigl(\delta_0,\ \lambda^{ac}(x_0)\,\Leb^d\restr Q,\ \mu_{x_0}\bigr)$ at regular points,
    \item $\sigma=\bigl(\delta_0,\ \lambda_{x_0},\ \mu_{x_0}\bigr)$  at singular points,
\end{enumerate}
where $\lambda^{ac}$ is the density of the absolutely continuous part of $\lambda$ and $\lambda_{x_0}\in\Measp(\overline Q)$ is a non-zero tangent measure of $\lambda^s$ at $x_0$. In both cases the direction part of $\sigma$ is the constant family $\mu_{x_0}$. At a regular point with $\lambda^{ac}(x_0)>0$, and at a singular point, Lemma~\ref{lem:theta} shows that $W_j\to0$ in measure and $\Theta_{W_j}\toweakstar c\,\mu_{x_0}$ in $\Meas(\partial B)$ with $c=\lambda^{ac}(x_0)$, respectively $c=\lambda_{x_0}(\overline Q)$, positive; Theorem~\ref{thm:VMC}, applied in $Q$ to the normalized fields $W_j/\|W_j\|_{L^1(Q)}$, then shows that $\mu_{x_0}$ is a $\Lambda$-superposition. Since the regular points with $\lambda^{ac}(x_0)=0$ form a $\lambda^{ac}$-null set, the conclusion holds for $\lambda$-a.e.\ $x_0\in\Omega$.
\end{proof}

\medskip
\noindent\textbf{Divergence-quasiconvexity.} A locally bounded Borel $h:\X\to\R$ is
\emph{$\Div$-quasiconvex} if $h(A)\le\int_{(0,1)^d}h(A+w(y))\,dy$ for every $(0,1)^d$-periodic
$w\in C^\infty(\R^d;\X)$ with $\Div w = 0$ and mean zero, and \emph{$\Lambda$-convex} if
$t\mapsto h(A+tM)$ is convex for every $A\in\X$ and $M\in\Lambda$. We use three facts, all of which
hold because $\Lambda$ is a closed symmetric cone spanning $\X$: a $\Div$-quasiconvex $h$ is
$\Lambda$-convex \cite{FonsecaMuller99}, \cite[Lem.~4.3]{ArroyoRabasa20}; a $\Lambda$-convex $h$ of
linear growth is globally Lipschitz \cite[Prop.~4.5(d)]{ArroyoRabasa20}, and its upper recession
function $h^{\#}$ is then real-valued, continuous, positively $1$-homogeneous and $\Lambda$-convex
\cite[Lem.~2.4]{KirchheimKristensen16};  $h^{\#}$ is \emph{convex at every point of $\Lambda$},
meaning that for each $X\in\Lambda$ there is a linear function $\rho_X\le h^{\#}$ on $\X$ with
$\rho_X(X) = h^{\#}(X)$ \cite[Thm.~1.1]{KirchheimKristensen16}. We shall also use the subadditivity
\begin{equation}\label{eq:subadd}
    h(A+M) \ \le\ h(A)+h^{\#}(M) , \qquad A\in\X,\ M\in\Lambda ,
\end{equation}
which follows by applying convexity of $t\mapsto h(A+tM)$ to the difference quotients
$h(A+M)-h(A)\le t^{-1}\bigl(h(A+tM)-h(A)\bigr)$, for $t\ge1,$ and letting $t\to\infty$.

\begin{lemma}\label{lem:jensen}
Let $\mu\in\Prob(\partial B)$ be a $\Lambda$-superposition. Then, for every $\Div$-quasiconvex
$h:\X\to\R$ of linear growth and every $s\ge0$, it holds
\begin{equation}\label{eq:jensen}
    h\bigl(s\,\bar\mu\bigr) \ \le\ h(0)+s\,\la\mu,h^{\#}\ra .
\end{equation}
\end{lemma}
\begin{proof}
Let $\pi$ and $\{\nu_X\}$ be as in \eqref{eq:superposition}. Since each $\nu_X$ is supported on the
unit sphere, we have $|X| = |\bar\nu_X|\le1$ for $\pi$-a.e.\ $X$. Moreover, up to replacing $\pi$ with $ \mathds 1_{\overline B} \pi$ (which again satisfies \eqref{eq:superposition} together with  $\{\nu_X\}$), we can assume throughout the proof that  $\pi$ is supported on the compact set
$\Lambda\cap\overline B$.  Iterating \eqref{eq:subadd} and using the $1$-homogeneity of $h^{\#}$ gives
\[
    h\Bigl(s\sum_{i=1}^k\alpha_
iX_i\Bigr)\ \le\ h(0)+s\sum_{i=1}^k\alpha_i\,h^{\#}(X_i)
\]
for $s\ge0$, $\alpha_i\ge0$ with $\sum_i\alpha_i = 1$ and $X_i\in\Lambda$. Since the measure $\pi$ is a
weak-$*$ limit of discrete probability measures on $\Lambda\cap\overline B$, on which $h$ and
$h^{\#}$ are continuous and bounded, passing to the limit yields
\begin{equation}\label{eq:step1}
    h\bigl(s\,\bar\pi\bigr)\ \le\ h(0)+s\int_\Lambda h^{\#}\,d\pi .
\end{equation}
Next, for $\pi$-a.e.\ $X$ one has $X = \bar\nu_X\in\Lambda$, so convexity of $h^{\#}$ at such points
$X$ gives $h^{\#}(X) = \rho_X(\bar\nu_X) = \la\nu_X,\rho_X\ra\le\la\nu_X,h^{\#}\ra$, having used that $\rho_X$ is linear function, smaller than $h^{\#}$ and agreeing at $X\in \Lambda$. Integrating in
$\pi$ and using $\mu = \int\nu_X\,d\pi$,
\begin{equation}\label{eq:step2}
    \int_\Lambda h^{\#}\,d\pi \ \le \ \int \la\nu_X,h^{\#}\ra\, d\pi(X)\ = \ \la\mu,h^{\#}\ra .
\end{equation}
Finally $\bar\pi = \int X\,d\pi(X) = \int\bar\nu_X\,d\pi(X) = \bar\mu$, and \eqref{eq:step1} together
with \eqref{eq:step2} is \eqref{eq:jensen}.
\end{proof}

\begin{corollary}\label{cor:approx}
Let $\bnu$ be as in \eqref{eq:choice}, with $\lambda$ and $\{\mu_x\}$ as in Theorem~\ref{thm:main}. Then $\bnu$ is a divergence-free generalized Young measure,
and there exists a divergence-free sequence $(w_k)_k \subset (C^1 \cap L^1)(\Omega;\X)$ with 
\[
w_k\toY\bnu .
\]
\end{corollary}

\begin{proof}
The barycenter $\bar\bnu = \bar \mu_x\,\lambda$ of $\bnu$ is divergence-free by (i) in Remark \ref{rem:localization}, and the concentration measure satisfies $\lambda(\partial\Omega) = 0$. (The theory of \cite{ArroyoRabasa20} is developed for bounded open sets with $\vol(\partial\Omega)=0$. This assumption is unnecessary here, since the Lebesgue part of $\bnu$ is $\Leb^d\restr\Omega$ and $\lambda(\partial\Omega)=0$.) Let $\lambda^{ac}$ denote the density of the absolutely continuous part of $\lambda$. The Jensen inequality required by \cite[Thm.~1.1]{ArroyoRabasa20} at $\Leb^d$-a.e.\ $x\in\Omega$ reads
\[
    h\bigl(\lambda^{ac}(x)\,\bar\mu_x\bigr)\ \le\ h(0)+\lambda^{ac}(x)\,\la\mu_x,h^{\#}\ra
\]
for every $\Div$-quasiconvex $h$ with linear growth. This is trivial on points where $\lambda^{ac}(x)=0$, and elsewhere it is \eqref{eq:jensen} with $s=\lambda^{ac}(x)$, because the set of points with $\lambda^{ac}(x)>0$ at which $\mu_x$ is not a $\Lambda$-superposition is $\lambda$-negligible, hence $\Leb^d$-negligible. The remaining condition of \cite[Thm.~1.1]{ArroyoRabasa20} is that $\bar\mu_x$ be carried by $\Span\Lambda=\X$ at $\lambda^s$-a.e.\ $x$, which is automatic. Thus, $\bnu$ is a divergence-free generalized Young measure and, by
\cite[Cor.~1.1]{ArroyoRabasa20}, is generated by divergence-free $\mu_j\in\Meas(\Omega;\X)$ with
uniformly bounded variation. For fixed $j$, \cite[Thm.~1.3]{ArroyoRabasa20} gives a divergence-free
$(w_{j,k})_k \subset L^1(\Omega;\X)$ with $w_{j,k}\Leb^d\toweakstar\mu_j$ and
$\mathrm{Area}(w_{j,k},\Omega)\to\mathrm{Area}(\mu_j,\Omega)$; as
$|\mu|(\Omega)\le\mathrm{Area}(\mu,\Omega)\le\vol(\Omega)+|\mu|(\Omega)$, both families are bounded
in variation. Since bounded sets of Young measures are metrizable, a diagonal argument yields a
subsequence with $w_{j(k),k}\toY\bnu$. That $(w_k)_k$ may be taken of class $C^1$ follows from~\cite[Thm.~1.3 and Rmk.~1.3]{ArroyoRabasa20} (the recovery sequence of that theorem can be taken of class $C^l$ for any $l$).
\end{proof}
Everything is thus reduced to the following statement, which is where the constraint has to be
repaired and which occupies the rest of the note.
\begin{lemma}[Modification lemma]\label{lem:modification}
Let $\Omega\subset\R^d$ be bounded, connected and open and let $(V_n)_n\subset C^1\cap L^1(\Omega;\X)$ satisfy
$\Div V_n = 0$ in $\mathcal D'(\Omega)$, $V_n\to0$ in measure, and $\int_\Omega|V_n|\le C$. Then
there are divergence-free fields $(U_n)_n\subset L^1(\Omega;\X)$ and open sets $A_n \subset \Omega$ such that
\[
    \|V_n-U_n\|_{L^1(\Omega)}\to0, \qquad \spt(U_n) \subset A_n,\qquad \vol\bigl(A_n\bigr)\to0 .
\]
Furthermore, we can also require that $\partial \Omega \subset \partial A_n$ and that $\partial_\Omega A_n$ is  a compact smooth manifold for all $n\in\mathbb N$.
\end{lemma}
Indeed, assuming the above statement, our main result follows.
\begin{proof}[Proof of Theorem~\ref{thm:main}]
Corollary~\ref{cor:approx} provides divergence-free $w_k\in L^1(\Omega;\X)$, bounded in $L^1$, with
$w_k\toY\bnu$; by Lemma~\ref{lem:theta} they converge to zero in measure.
Lemma~\ref{lem:modification} replaces them by divergence-free $U_k$ with $\|w_k-U_k\|_{L^1}\to0$ and $\vol(A_k)\to0$.

Notice that such a modification does not change the generated Young measure
\cite[Prop.~4.1]{ArroyoRabasa20}, so $U_k\toY\bnu$ and
$\bTheta_{U_k}\toweakstar\lambda\otimes\mu_x$ by Lemma~\ref{lem:theta} again, which also gives $\|U_k\|_{L^1}\to\lambda(\Ob) = 1$,
so $U_k\ne0$ for $k$ large and $V_k\coloneqq U_k/\|U_k\|_{L^1}$ is well defined, divergence-free and
supported where $U_k$ is. Since $\bTheta_V$ is positively $1$-homogeneous in $V$, we still get
$\bTheta_{V_k}\toweakstar\lambda\otimes\mu_x$, and $(V_k)_k$ satisfies \eqref{eq:vanishing}. Finally, the last conclusion is obtained by the last conclusion of Lemma \ref{lem:modification}.
\end{proof}

%=====================================================================
\section{Filling by currents}\label{sec:filling}
%=====================================================================

The goal of this section is to show Lemma \ref{lem:modification}. Throughout this section the fields are vector valued ($m=1$). The matrix case is obtained by treating the
$m$ rows separately.

\subsection{Truncation}

The elementary first step is that almost all of the $L^1$ mass of a sequence converging to zero in
measure already sits on a set of small volume.

\begin{proposition}\label{prop:truncation}
Let $(v_n)_n\subset C^1(\Omega;\X)$ satisfy $v_n\to0$ in measure and $\int_\Omega|v_n|\le C<\infty$.
Then there are sets $A_n\subset\Omega$ and real numbers $t(n)\to\infty$ such that, for all
$n$ sufficiently large,
\[
    \|v_n-v_n\mathds 1_{A_n}\|_{L^1}\le t(n)^{-1},
    \qquad |v_n|\ge t(n)\ \text{on }A_n,
    \qquad \vol(A_n)\le C\,t(n)^{-1}.
\]
Moreover, the sets $A_n$ can be chosen relatively closed in $\Omega$, of locally finite
perimeter in $\Omega$, and with $\vol(\partial_\Omega A_n)=0$.
\end{proposition}
\begin{proof}
Fix $j\in\mathbb N$. Since $v_n\to0$ in measure and $|v_n\mathds 1_{\{|v_n|<j\}}|\le j\in L^1(\Omega)$,
dominated convergence gives $\int_{\{|v_n|<j\}}|v_n|\to0$ as $n\to\infty$. Choose
$N_1<N_2<\cdots$ with $\int_{\{|v_n|<j+1\}}|v_n|\le\frac1{j+1}$ for all $n\ge N_j$, and set
$A_n\coloneqq\{|v_n|\ge t(n)\}$ with $t(n)\in[j,j+1)$ for $N_j\le n<N_{j+1}$. Set 
$t(n)\in[1,2)$ for $n<N_1$, so that $t(n)\to\infty$. The precise value of $t(n)$ is fixed
below. For $N_j\le n<N_{j+1}$ the first assertion follows from
\[
    \|v_n-v_n\mathds 1_{A_n}\|_{L^1}=\int_{\{|v_n|<t(n)\}}|v_n|
    \le\int_{\{|v_n|<j+1\}}|v_n|\le\frac1{j+1}\le t(n)^{-1}.
\]
Finally, $A_n$ is relatively closed in $\Omega$ because $|v_n|$ is continuous, and
$\partial_\Omega A_n\subset\{|v_n|=t(n)\}$. Since the level sets $\{|v_n|=t\}$ are pairwise
disjoint (here we use continuity), $\vol(\{|v_n|=t\})>0$ for at most countably many $t$. Since $|v_n|$ is
locally Lipschitz, the coarea formula applied on an exhaustion of $\Omega$ shows that
$\{|v_n|>t\}$ has locally finite perimeter in $\Omega$ for a.e.\ $t$. Choosing $t(n)$ in
the intervals above outside these two exceptional sets, we get $\vol(\partial_\Omega A_n)=0$. In particular, $A_n$ differs from $\{|v_n|>t(n)\}$ by a null set, hence it has locally finite perimeter
in $\Omega$.
\end{proof}
\subsection{The Kantorovich--Rubinstein norm}\label{ssec:KR}
In this section, we shall introduce the Kantorovich-Rubinstein norm that, for our scope, will measure the cost of the filling in Lemma \ref{lem:filling}. We first introduce the general metric space set-up and introduce the so-called Arens-Eells space (we refer to \cite{Weaver18} for a modern treatment). Given a metric space $(X,d)$, the Arens-Eells space $\textnormal{\AE}((X,d))$ is the completion of the set of molecules
\[
    q = \sum_{i=1}^N \lambda_i(\delta_{a_i}-\delta_{b_i}),\qquad \lambda_i \in \R, a_i,b_i \in X, N\in \mathbb N,
\]
with respect to the norm
\[
    \rho(q) :=  \inf\left\{ \sum_i|\lambda_i|d(a_i,b_i) \ : \ q = \sum \lambda_i(\delta_{a_i}-\delta_{b_i})\right\},
\]
where the infimum is taken among all possible representations of the molecule $q$. There is a natural coupling between a Lipschitz function $u :  X \to \R$ and a molecule $q \in \textnormal{\AE}((X,d))$, defined by
\begin{equation}
    u \mapsto \langle u,q\rangle := \int u \, d q = \sum_i \lambda_i ( u(a_i)-u(b_i)),\label{eq:pairing}
\end{equation}
which is well-defined and independent of the representation of $q$. In fact, it is a linear functional that extends as an isometry between the dual of $\textnormal{\AE}((X,d))$ (as Banach space with $\rho$) and the space of Lipschitz functions modulo constants; the space $\textnormal{\AE}((X,d))$ itself is also known as the Lipschitz-free space of $X$. Arens and Eells (\cite{ArensEells56}) proved that $\text{\AE}((X,d))$ is a predual of $\Lip_0(X)$, the space of all Lipschitz maps $u: X \to \R$ on $X$ vanishing at a special point $x_o \in X$, endowed with the (semi)norm
\[
    \Lip(u) \coloneqq \sup_{x \neq y} \frac{|u(x) - u(y)|}{d(x,y)} \, .
\]
The fact that $\text{\AE}((X,d))$ is the completion of dipoles, which are measures with zero average, conveys that every $f \in \text{\AE}((X,d))$ can be extended uniquely to a continuous linear form (still denoted $f$) on the whole of $\Lip(X)$ satisfying 
\[
    \la f, u\ra \le C \Lip(u), \qquad \la f,1\ra = 0. 
\]
The least constant $C>0$ above is given by the \emph{Kantorovich--Rubinstein norm}
\[
\rho(f) = \|f\|_\text{KR} = \sup_{u \in \Lip_1(X)} \la f,u\ra,
\]
where $\Lip_1(X)$ is the set of all Lipschitz functions with $\Lip(u) \le 1$. Any $\mu \in \mathcal M(X)$ with $\mu(X)=0$ and finite first moment can be identified with an element of $\text{\AE}(X)$ and so its KR-norm  can be written as
\[
\|\mu\|_{\mathrm{KR}} = W_1(\mu^+,\mu^-) = \sup_{u \in \Lip_1(X)}  \int u\, d\mu 
\]
where $W_1(\mu^-,\mu^+)$ denotes the Wasserstein-$1$ distance from $\mu^-$ to $\mu^+$, where $\mu = \mu^+ - \mu^-$ is the Hahn decomposition of $\mu$. We refer to \cite[Ch.\ 3]{Weaver18} for a complete discussion.

We next specialize this discussion to the metric space given by an open connected set \(\Omega\), endowed with the intrinsic distance. The cost of a filling has to be measured \emph{inside} $\Omega$, accordingly, let
\begin{equation}\label{eq:geod}
    d_\Omega(a,b) \ \coloneqq\ \inf\bigl\{\ \operatorname{length}(\gamma)\ :\
    \gamma\text{ a Lipschitz path in }\Omega\text{ from }a\text{ to }b\ \bigr\} .
\end{equation}
be the intrinsic distance of $\Omega$.

\begin{remark}\label{rem:polygonal} Polygonal paths suffice in \eqref{eq:geod}: a Lipschitz path
$\gamma$ in $\Omega$ is a compact subset of $\Omega$, hence at some positive distance $\varrho$ from
$\partial\Omega$, and the polygon inscribed in $\gamma$ along a partition fine enough that
consecutive vertices are less than $\varrho$ apart lies in $\Omega$ -- each of its segments lies in
a ball of radius $\varrho$ centred on (the image of) $\gamma$; it is not longer than $\gamma$.
\end{remark}

Notice that
$d_\Omega$ agrees with $|\cdot|$ on every ball contained in $\Omega$. This has three
consequences. First, $d_\Omega$ is a finite distance  thanks to the connectedness of $\Omega$. Indeed for any $a\in \Omega$, the set of points joined to $a$ by a polygonal path in $\Omega$ contains $a$, is open and its complement in
$\Omega$ is open as well. Second, together
with the triangle inequality, $d_\Omega$ is continuous on $\Omega\times\Omega$ and induces the
Euclidean topology. Third, $d_\Omega$ is bounded on $K\times K$ for every compact $K\Subset\Omega$ by continuity, although it need not be bounded on
$\Omega\times\Omega$: a bounded open set may
spiral so as to have infinite intrinsic diameter. Only the bound on compact sets will be used.
For $u:\Omega\to\R$ we write
\[
\Lip_{d_\Omega}(u) \ \coloneqq\ \sup \ \left\{ \ \frac{|u(x)-u(y)|}{d_\Omega(x,y)} \ : \  x\ne y\ \right\}
\]
for
its Lipschitz constant with respect to $d_\Omega$, the quotient being $0$ when
$d_\Omega(x,y) = +\infty$. Notice that a locally Lipschitz $u$ has $\Lip_{d_\Omega}(u)\le1$ if and only if
$|\nabla u|\le1$ almost everywhere. This follows immediately from the fact that $d_\Omega$ and the Euclidean metric coincide on sufficiently small balls in $\Omega$ and $d_\Omega(x,y)\ge |x-y|.$

Let now $f = \Div v$ be a load, with $v\in L^1(\Omega;\R^d)$ compactly supported. We define its Kantorovich--Rubinstein norm relative to $\Omega$ by
\begin{equation}\label{eq:KR}
    \fl{f} \ = \ \sup\bigl\{\la f,u\ra\ :\ u:\Omega\to\R,\ \Lip_{d_\Omega}(u)\le1\bigr\} ,
\end{equation}
and notice that
\[
 \la f,u\ra \ = \ -\int_\Omega v\cdot\nabla u\,dx \,.
\]
The integral being well defined by Rademacher's theorem, and by the assumption that $\spt v \Subset \Omega$: it depends on $f$ only, since if $\Div v=\Div v'$ then $\int_\Omega (v-v')\cdot\nabla u\,dx=\int_\Omega (v-v')\cdot\nabla(\chi u)\,dx=0$ for any $\chi\in C^\infty_c(\Omega)$ with $\chi=1$ near $\spt v\cup\spt v'$. Proposition~\ref{prop:AE} below shows that $f$ can be identified with an element of $\text{\AE}((K,d_\Omega))$ for a compact $K\Subset\Omega$, and that \eqref{eq:KR} is its Arens--Eells norm. For a finite measure $\mu$ on $\Omega$ with compact support and $\mu(\Omega)=0$ we write accordingly $\fl{\mu}\coloneqq\sup\{\int u\,d\mu : \Lip_{d_\Omega}(u)\le1\}$, which is finite and consistent with \eqref{eq:KR} when $\mu=\Div v$. This is the meaning of $\fl{\cdot}$ in \eqref{eq:dipole} and in Lemma~\ref{lem:refine}.

Note that $\la f,1\ra = 0$, which means that $f$ has zero average in $\Omega$. The
constraint $|\nabla u|\le1$ gives the \emph{a priori bound}
\begin{equation}\label{eq:apriori}
    \fl{\Div v} \ \le\ \int_\Omega|v|\,dx .
\end{equation}

For $a\ne b$ {in $\R^d$} let $\sigma_{a,b}\coloneqq\tau\,\Haus^1\restr\llbracket a,b\rrbracket$, with
$\tau = (b-a)/|b-a|$, be the unit field tangent to the segment, {regarded as a measure on the whole
of $\R^d$ (the segment is not required to lie in $\Omega$)}. A direct computation gives{, in
$\mathcal D'(\R^d)$,}
\begin{equation}\label{eq:segment}
    \Div\sigma_{a,b} \ = \ \delta_a-\delta_b ,
    \qquad |\sigma_{a,b}|(\R^d) \ = \ |b-a| .
\end{equation}
Concatenating segments along a polygonal path $\gamma\subset\Omega$ from $a$ to $b$ produces a field
$\sigma_\gamma$, now supported in $\Omega$. For this piecewise affine path it holds $\Div\sigma_\gamma = \delta_a-\delta_b$ {in
$\mathcal D'(\Omega)$} and $|\sigma_\gamma|({\Omega}) \le \operatorname{length}(\gamma)$, with equality unless some of its segments overlap with opposite orientations. Testing \eqref{eq:KR} with $u = -d_\Omega(a,\cdot)$, gives
\begin{equation}\label{eq:dipole}
    \fl{\delta_a-\delta_b} \ = \ d_\Omega(a,b) .
\end{equation}
Therefore, dipoles are the elementary loads, and their cost is the geodesic distance. That every
load of finite norm is a countable superposition of dipoles of almost optimal total cost is the
classical content of the Arens--Eells duality:

\begin{proposition}[Arens--Eells]\label{prop:AE}
Let $\Omega\subset\R^d$ be a bounded and connected open set, let $v\in L^1(\Omega;\R^d)$ have compact support in
$\Omega$, and let $f = \Div v$. Then $f\in\text{\AE}((\Omega,d_\Omega))$, and for every $\varepsilon>0$ there exist a compact set $K\Subset\Omega$, points $a_i,b_i\in K$ and weights ${\lambda_i}>0$, $i\in\mathbb N$, with
\[
    f \ = \ \sum_i{\lambda_i}\bigl(\delta_{a_i}-\delta_{b_i}\bigr) \quad\text{in }\mathcal D'(\Omega),
    \qquad \sum_i{\lambda_i}\,d_\Omega(a_i,b_i) \ \le\ \fl f+\varepsilon ,
\]
the first series converging in the norm \eqref{eq:KR}.
\end{proposition}

\begin{proof} 
We subdivide the proof into two steps.

\emph{Step 1: $f\in\textnormal{\AE}((K,d_\Omega))$ for some compact $K\Subset\Omega$.} 
By \eqref{eq:dipole} {and Arens--Eells duality}, $\fl\cdot$ is the
Arens--Eells norm $\textnormal{\AE}((K,d_\Omega))$ built on the compact metric space $(X,d)=(K,d_\Omega)$ -- compact because $K$ is -- for every
compact $K\subset\Omega$.
Two remarks on this identification.
The distance $d_\Omega$ is restricted to $K$, but its paths still run in $\Omega$: nothing is asked of
$K$ beyond compactness -- it may well be totally disconnected -- and \eqref{eq:dipole} applies to
its points because they lie in $\Omega$. Moreover, every $d_\Omega$-$1$-Lipschitz function on $K$ is
the restriction of one on $\Omega$, by McShane's formula
$u\mapsto\inf_{y\in K}\bigl(u(y)+d_\Omega(\cdot,y)\bigr)$, so for a load supported in $K$ the supremum
in \eqref{eq:KR} coincides with the supremum over Lipschitz functions on $K$ that defines the dual
norm of $\textnormal{\AE}((K,d_\Omega))$.

Mollifying $v$ at a small scale, choose $v_j\in C^\infty_c(\Omega;\R^d)$ with $v_j\to v$ in $L^1$ and all supports in a fixed compact
$K\Subset\Omega$; by \eqref{eq:apriori}, $\fl{f-\Div v_j}\le\|v-v_j\|_{L^1}\to0$. By Gau\ss--Green,
$\Div v_j$ is a smooth compactly supported function of zero integral. Partition $K$ into finitely
many Borel pieces $\{{Q_k}\}_k$ of $d_\Omega$-diameter at most $\varsigma>0${, which is possible
because $(K,d_\Omega)$ is compact}, and replace $\Div v_j$ on each piece by a Dirac mass
${\lambda_{j,k}}\,\delta_{x_{j,k}}$ of the same total flux ${\lambda_{j,k}} = \int_{{Q_k}}\Div v_j\,dx$ at
a point $x_{j,k}\in {Q_k}$. The result $f_j = \sum_k{\lambda_{j,k}}\,\delta_{x_{j,k}}$ is a finitely
supported measure of zero total mass, hence it can be represented by a finite sum of weighted dipoles $\delta_a-\delta_b$
with endpoints {$a,b\in K$} (a molecule), hence 
$f_j\in\textnormal{\AE}(K)$ and $\rho(f_j) = \fl{f_j}$ by the Arens-Eells duality. Moreover $\fl{\Div v_j-f_j}\le\varsigma\|\Div v_j\|_{L^1}$, because a
$d_\Omega$-$1$-Lipschitz test function $u$ varies by at most $\varsigma$ on each element of the
partition:
\[
    \int_{{Q_k}}u\,{d(\Div v_j-f_j)}
    \ = \ \int_{{Q_k}}{\bigl(u-u(x_{j,k})\bigr)\Div v_j\,dx}
    \ \le\ \varsigma\int_{{Q_k}}|\Div v_j|\,dx .
\]
Choosing $\varsigma=\varsigma_j$ with $\varsigma_j\|\Div v_j\|_{L^1}\to0$ and letting $j\to\infty$, the molecules $f_j$ converge to $f$ in the norm \eqref{eq:KR}; hence $f\in\textnormal{\AE}((K,d_\Omega))$, with $\rho(f)=\fl f$. In particular, we also have $f\in\textnormal{\AE}((\Omega,d_\Omega))$.

\emph{Step 2: an exact decomposition.} Put $\varepsilon'\coloneqq\varepsilon/8$. By Step~1 there is a
finite sum of dipoles $q_1$ with $\fl{f-q_1}\le\varepsilon'/2$ and $\rho(q_1)\le\fl f+\varepsilon'/2$.
Applying the same to $f-q_1$ with $\varepsilon'/4$ in place of $\varepsilon'/2$, and iterating,
produces finite sums $q_j$ with $\fl{f-q_1-\cdots-q_j}\le2^{-j}\varepsilon'$ and, for $j\ge2$,
$\rho(q_j)\le\fl{f-q_1-\cdots-q_{j-1}}+2^{-j}\varepsilon'\le2^{-j+2}\varepsilon'$. Hence
$\sum_j\rho(q_j)\le\fl f+\varepsilon'/2+2\varepsilon'\le\fl f+\varepsilon$, and concatenating the
dipoles of the $q_j$ gives the assertion.
\end{proof}

\begin{remark}\label{rem:metric}
Everything up to here uses only that $(\Omega,d_\Omega)$ is a metric space, in the sense in which
Arens--Eells duality is developed in \cite{Weaver18}; the Euclidean structure of open sets in $\R^d$ enters only
in \S\ref{ssec:thicken}, which needs room around a segment for the tubes. Proposition~\ref{prop:AE}
is the elementary, polygonal face of a general statement: in a separable, piecewise quasiconvex
metric space with quasiconvexity constant $c\ge1$, every molecule $q$ bounds a rectifiable normal
$1$-current of mass at most $c(1+\delta)\|q\|_{\mathrm{KR}}$, $\delta\in(0,1)$
\cite[Lem.~4.1]{ArroyoRabasaBouchitte25}; here $(\Omega,d_\Omega)$ is a length space, so $c = 1$ and
the filling may be taken polygonal, as the thickening requires.
\end{remark}

The following is a refined version of the Arens--Eells lemma, tailored for measures supported on a compact set. The construction will be crucial in obtaining the existence of vanishing regular sequences.

\begin{lemma}\label{lem:refine}
Let $\Omega\subset\R^d$ be open and connected, let $K\subset\Omega$ be compact, and let
$\mu\in\Meas(K)$ be such that $\mu(K)=0$. Then for every
$\varepsilon>0$ there are points $a_i,b_i\in K$, weights
$\lambda_i\in\R\setminus\{0\}$ and polygonal paths $\gamma_i\subset\Omega$
from $a_i$ to $b_i$, $i\in\mathbb N$, such that
\begin{enumerate}
  \item\label{it:i} $\mu=\sum_i\lambda_i(\delta_{a_i}-\delta_{b_i})$ in
  $\mathcal D'(\Omega)$, the series converging absolutely in the norm
  \eqref{eq:KR};
  \item\label{it:ii} $\sum_i|\lambda_i|\,\ell(\gamma_i)\le\fl{\mu}+\varepsilon$;
  \item\label{it:iii} the set $L\coloneqq K\cup\bigcup_i\gamma_i$ is compact
  and $\vol(L)=\vol(K)$.
\end{enumerate}
\end{lemma}

\begin{remark}
The points \ref{it:i}--\ref{it:ii} can be inferred from Proposition~\ref{prop:AE} for measures. The refinement is about their compatibility with \ref{it:iii}. The paths can be chosen to accumulate only on
$K$ avoiding, for example, density of the paths in $\Omega$. In particular, if $\vol(K)=0$, then $L$ is a compact negligible set,
which is what the last conclusion of Lemma~\ref{lem:filling} requires.
\end{remark}
\begin{proof}
We may assume $|\mu|(K)>0$ for otherwise the statement is trivial. Let $r\coloneqq\dist(K,\Omega^c) > 0$
and fix a decreasing sequence $\delta_m\downarrow0$ with
\begin{equation}\label{eq:deltachoice}
    \delta_1<\tfrac r2,\qquad
  \sum_{m\ge1}\delta_m\ \le\ \frac{\varepsilon}{4\,|\mu|(K)} .
\end{equation}
We will often use the following:
if $x\in K$ and $|y-x|<r$, then $\llbracket x,y\rrbracket\subset B_r(x)\subset\Omega$, hence
\begin{equation}\label{eq:equiv}
  d_\Omega(x,y)=|x-y| .
\end{equation}
Note also that every $\varphi\in C^\infty_c(\Omega)$ satisfies
$\Lip_{d_\Omega}(\varphi)\le\|\nabla\varphi\|_\infty$ (see \S\ref{ssec:KR}), so that
convergence in the norm \eqref{eq:KR} implies convergence in $\mathcal D'(\Omega)$.

\emph{Step 1: Nested nets.} For each  $m\in \mathbb N$ we build a finite Borel partition
$\mathcal P_m$ of $K$ satisfying the following properties:
\begin{enumerate}[label=\roman*)]
    \item for all $m\in \mathbb N$ the elements of $\mathcal P_m$ have  diameter at most $\delta_m$;
    \item  for all $m\in \mathbb N$ and all $Q\in \mathcal P_{m+1}$, there exists a unique $F(Q)\in \mathcal P_m$ containing $Q$ called \textit{parent of $Q$}.
\end{enumerate}
 This can be done as follows.  First for each $m\in \mathbb N$ construct a partition $ \mathcal P'_m$ by covering $K$ by finitely many balls $\{B_i\}_{i=1}^{N_m}$ of radius $\delta_m/2$ and then setting $ \mathcal P'_m\coloneqq \cup_{i=1}^{N_m} U_i$, where $U_i$ are  inductively defined by
 \[
 U_1\coloneqq B_1\cap K; \quad U_i\coloneqq (B_i\cap K)\setminus (U_1\cup \dots U_{i-1}), \quad i=2,\dots, N_m.
 \]
 Then we set $\mathcal P_1\coloneqq \mathcal P'_1$ and inductively 
 $$\mathcal P_m\coloneqq \bigcup_{{{Q\in \mathcal P'_m,\, R\in \mathcal P_{m-1}}},\, \,{Q\cap R\neq \emptyset} } Q\cap R, \quad m>1. $$
The families $\mathcal P_m$ have the properties  stated above.

For each $Q\in\mathcal P_m$ fix a ``center'' $x_Q\in Q$. Then, by \eqref{eq:equiv},
\begin{enumerate}[label=(\alph*)]
  \item\label{a1} $d_\Omega(x,x_Q)\le\delta_m$ for all $ x\in Q$,
  \item\label{b1}$d_\Omega(x_Q,x_{F(Q)})=|x_Q-x_{F(Q)}|\le\delta_{m-1}$,  provided $m\ge2$,
  \item\label{c1} $\llbracket x_Q,x_{F(Q)}\rrbracket\subset\Omega\cap\{\dist(\cdot,K)\le\delta_{m-1}\}$, provided $m\ge2$.
\end{enumerate}

\emph{Step 2: $m$-scale atomic approximations of $\mu$.} Set
\[
\mu_m\coloneqq\sum_{Q\in\mathcal P_m}\mu(Q)\,\delta_{x_Q},
\]
which has zero mass on  $\Omega$, which is necessary for $\fl{\mu_m} < \infty$. If $u$ is $1$-Lipschitz for $d_\Omega$, then 
\[
  \int u\,d(\mu-\mu_m)
  =\sum_{Q\in\mathcal P_m}\int_Q\bigl(u(x)-u(x_Q)\bigr)\,d\mu(x)
  \le\delta_m\,|\mu|(K)
\]
so that 
\[
\fl{\mu-\mu_m}\le\delta_m|\mu|(K)\longrightarrow 0.
\]
Moreover, since $\mathcal P_{m+1}$ refines $\mathcal P_m$, we have
$\mu_m=\sum_{Q\in\mathcal P_{m+1}}\mu(Q)\,\delta_{x_{F(Q)}}$ and hence
\begin{equation}\label{eq:telescope}
  \mu_{m+1}-\mu_m=\sum_{Q\in\mathcal P_{m+1}}\mu(Q)\bigl(\delta_{x_Q}-\delta_{x_{F(Q)}}\bigr).
\end{equation}
Moreover
\[
\sum_{Q\in\mathcal P_{m+1}}|\mu(Q)|\,d_\Omega(x_Q,x_{F(Q)})\le\delta_m\sum_{Q\in\mathcal P_{m+1}} |\mu|(Q)\le \delta_m |\mu|(K).
\]
Summing over $m$ and using
\eqref{eq:deltachoice},
\begin{equation}\label{eq:tail}
  \sum_{m\ge2}\ \sum_{Q\in\mathcal P_m}|\mu(Q)|\,d_\Omega(x_Q,x_{F(Q)})
  \ \le\ |\mu|(K)\sum_{m\ge1}\delta_m\ \le\ \frac\varepsilon4 .
\end{equation}
In particular the following series converges absolutely in
the KR norm 
\begin{equation}\label{eq:decomp}
  \mu=\mu_1+\sum_{m\ge2}\sum_{Q\in\mathcal P_m}\mu(Q)\bigl(\delta_{x_Q}-\delta_{x_{F(Q)}}\bigr) 
\end{equation}
and again by \eqref{eq:deltachoice}
\[
\fl{\mu_1}\le\fl\mu+\delta_1|\mu|(K)\le\fl\mu+\frac\varepsilon4.
\]

\emph{Step 3: Transport optimality in the first level.} The measure $\mu_1$ is supported on the
finite set $N_1\coloneqq\{x_Q:Q\in\mathcal P_1\}$ and has zero average. As we already discussed before, by McShane's formula, every $d_\Omega$-$1$-Lipschitz function on $N_1$ is the restriction of one on $\Omega$. Therefore, $\fl{\mu_1}$ coincides with the Kantorovich--Rubinstein norm
of $\mu_1$ on the metric space $(N_1,d_\Omega)$. Since $N_1$ is a finite metric space, the optimal transport plan between
$\mu_1^+$ and $\mu_1^-$ exists: there are pairs $x_j\ne y_j$ in $N_1$ and
$\lambda_j>0$, $j=1,\dots,N$, with
\begin{equation}\label{eq:mu1}
      \mu_1=\sum_{j=1}^N\lambda_j\bigl(\delta_{x_j}-\delta_{y_j}\bigr),
  \qquad
  \sum_{j=1}^N\lambda_j\,d_\Omega(x_j,y_j)=\fl{\mu_1} .
\end{equation}
By Remark~\ref{rem:polygonal} we may pick polygonal paths
$\gamma_j\subset\Omega$ from $x_j$ to $y_j$ with
$\ell(\gamma_j)\le d_\Omega(x_j,y_j)+2^{-j}\varepsilon/(4\lambda_j)$, so that 
\begin{equation}\label{eq:firstcost}
  \sum_{j=1}^N\lambda_j\,\ell(\gamma_j)\ \le\ \fl{\mu_1}+\frac\varepsilon4
  \ \le\ \fl\mu+\frac\varepsilon2 .
\end{equation}

\emph{Step 4: Conclusion.} Consider the sequence of quadruples $(a_i,b_i,\lambda_i,\gamma_i)_{i\in\mathbb N}$ obtained
by listing first the $N$ quadruples $(x_j,y_j,\lambda_j,\gamma_j)$ of Step~3 and then, step by step in $m\ge2$, the quadruples $\bigl(x_Q,x_{F(Q)},\mu(Q),\llbracket x_Q,x_{F(Q)}\rrbracket\bigr)$
for those $Q\in\mathcal P_m$ with $\mu(Q)\ne0$ and $x_Q\ne x_{F(Q)}$. Crucially, each level $m$ contributes
finitely many. Then, point \ref{it:i} follows from
\eqref{eq:decomp} and \eqref{eq:mu1}. Point \ref{it:ii} follows from \eqref{eq:firstcost},
\eqref{eq:tail} and the fact that, by \ref{b1}, the segment
$\llbracket x_Q,x_{F(Q)}\rrbracket$ has length $d_\Omega(x_Q,x_{F(Q)})$. Finally, to prove \ref{it:iii} let us write $L=K\cup\Gamma\cup S$
with
\[
\Gamma\coloneqq\bigcup_{j\le N}\gamma_j \quad \text{ and } \quad 
S\coloneqq\bigcup_{m\ge2}\bigcup_{Q\in\mathcal P_m}\llbracket x_Q,x_{F(Q)}\rrbracket.
\]
Clearly, both $\Gamma$
and $S$ are countable unions of segments, so $\vol(L)=\vol(K)$. 

On the other hand, $L$ is
bounded because $\Gamma$ is compact and $S$ is contained in the (bounded) enlargement $\{\dist(\cdot,K)\le\delta_1\}$
by \ref{c1}. From this and \ref{c1} again it follows $L \subset \Omega$. It remains
to show that $L$ is closed. Let $p_k\in L$ with $p_k\to p$. If infinitely
many $p_k$ lie in the closed set $K\cup\Gamma$, then $p\in K\cup\Gamma$.
Hence, we may assume $p_k\in S$ for all $k$. Each scale $m$ contributes of
finitely many segments to $S$. If the $p_k$ belong to only finitely many of them,
$p$ lies on one of these closed segments. If not, passing to a subsequence
we have $p_k\in\llbracket x_{Q_k},x_{F(Q_k)}\rrbracket$ with
$Q_k\in\mathcal P_{m_k}$ and $m_k\to\infty$. This requires
$\dist(p_k,K)\le\delta_{m_k-1}\to0$ by \ref{c1}, and therefore $p\in K$
($K$ is closed). Hence $L$ is compact.
\end{proof}

\subsection{Thickening a segment}\label{ssec:thicken}

A segment carries its mass on a set of zero volume, so the filling produced by
Proposition~\ref{prop:AE} is not yet a field. It is made into one by \emph{thickening}: each segment
is replaced by an average of two-legged paths through a small transverse disc, which leaves the
endpoints, and therefore the load, untouched {(Figure~\ref{fig:thicken})}.

\begin{figure}[t]
\centering
\begin{tikzpicture}[x=1cm, y=1cm, scale=1.15, >={Latex[length=2.2mm, width=1.8mm]}]
  \coordinate (a) at (-3.0,0);
  \coordinate (b) at ( 3.0,0);
  \coordinate (p) at (0,0);
  \coordinate (t) at (0,1.3);
  \coordinate (s) at (0,-1.3);
  \coordinate (y1) at (0.05,0.78);
  \coordinate (y2) at (0.20,-0.55);
  % the double cone: the support of the thickened segment
  \fill[gray!18] (a) -- (t) arc (90:270:0.38 and 1.3) -- cycle;
  \fill[gray!18] (b) -- (t) arc (90:-90:0.38 and 1.3) -- cycle;
  \draw[gray!70] (a) -- (t) (a) -- (s) (b) -- (t) (b) -- (s);
  % the transverse disc D_eps, seen in perspective
  \draw[thick] (p) ellipse (0.38 and 1.3);
  % the original segment
  \draw[dashed] (a) -- (b);
  % two-legged paths a -> y -> b, arrowheads in the middle of each leg
  \foreach \y in {y1, y2} {
    \draw[marine, thick] (a) -- (\y) -- (b);
    \draw[marine, thick, ->] (a) -- ($(a)!0.55!(\y)$);
    \draw[marine, thick, ->] (\y) -- ($(\y)!0.55!(b)$);
    \fill[marine] (\y) circle (1.3pt);
  }
  % labels
  \fill (a) circle (1.5pt) node[left] {$a$};
  \fill (b) circle (1.5pt) node[right] {$b$};
  \fill (p) circle (1.2pt) node[below right, inner sep=1pt] {\small $p$};
  \node[above] at (t) {$D_\varepsilon$};
  \node[marine, above left, inner sep=1.5pt] at (y1) {\small $y$};
  \node[below] at (0,-1.45) {$L = |b-a|$};
\end{tikzpicture}
\caption{(Thickening a segment). The field $\sigma^\varepsilon_{a,b}$ is a weighted superposition of the $1$-currents obtained concatenating paths from $a$ to $b$, passing through any $y \in D_{\varepsilon}$, and weighted by the normalized surface measure $d\mu_\varepsilon(y)$ of the transverse disc $D_{\varepsilon}$. The support of $\sigma^\varepsilon_{a,b}$ is the  shaded double cone of volume $O(\varepsilon^{d-1})$, whereas its boundary is given by $\Div\sigma^\varepsilon_{a,b} = \delta_a-\delta_b$. Intermediate boundary masses at $y$ cancel out by the concatenation procedure. (Image generated with the assistance of AI tools.)}
\label{fig:thicken}
\end{figure}
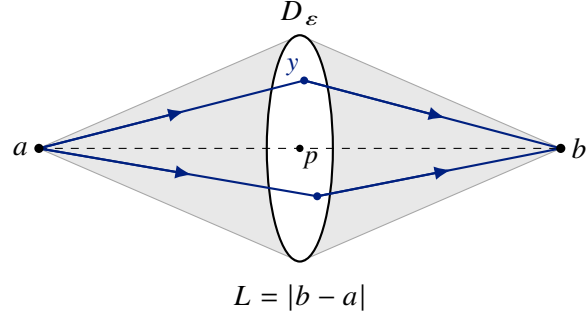

\begin{lemma}[Thickening]\label{lem:thicken}
Let $a\ne b$ in $\R^d$, let $L\coloneqq|b-a|$, let $p\coloneqq\frac{a+b}2$ and let $D_\varepsilon$
be the $(d-1)$-dimensional disc of radius $\varepsilon$ centred at $p$ and orthogonal to $b-a$, with
$\mu_\varepsilon$ its normalized surface measure. Then
\[
    \sigma^\varepsilon_{a,b} \ \coloneqq\ \int_{D_\varepsilon}
      \bigl(\sigma_{a,y}+\sigma_{y,b}\bigr)\,d\mu_\varepsilon(y)
\]
is an $L^1$ vector field, supported in the union of the two solid cones with apexes $a$, $b$ and
common base $D_\varepsilon$, and
\[
    \Div\sigma^\varepsilon_{a,b} = \delta_a-\delta_b ,
    \qquad \int\bigl|\sigma^\varepsilon_{a,b}\bigr| \ \le\ L+\frac{2\varepsilon^2}{L} ,
    \qquad \vol\bigl(\{\sigma^\varepsilon_{a,b}\ne0\}\bigr) \ \le\ \frac{{\omega_{d-1}}}{d}\,\varepsilon^{d-1}L ,
\]
where ${\omega_{d-1}}$ is the volume of the unit ball of $\R^{d-1}$.
\end{lemma}

\begin{proof}
By \eqref{eq:segment}, $\Div\int\sigma_{a,y}\,d\mu_\varepsilon(y) = \delta_a-\mu_\varepsilon$ and
$\Div\int\sigma_{y,b}\,d\mu_\varepsilon(y) = \mu_\varepsilon-\delta_b$; adding the two gives the
first identity. For the mass, the triangle inequality gives
\begin{align*}
    |\sigma^\varepsilon_{a,b}|(\R^d) 
    \ &\le\ \int_{D_\varepsilon}\bigl(|y-a|+|b-y|\bigr)\,d\mu_\varepsilon(y) \\
    \ & =\ \int_{D_\varepsilon}2\sqrt{\tfrac{L^2}4+|y-p|^2}\ d\mu_\varepsilon(y)
    \ \le\ \sqrt{L^2+4\varepsilon^2} \ \le\ L+\frac{2\varepsilon^2}{L} .
\end{align*}
The support is contained in the two cones, whose total combined volume is
$\frac1d\,{\omega_{d-1}}\varepsilon^{d-1}\cdot L$.

It remains to see that $\sigma^\varepsilon_{a,b}\ll\Leb^d$ with an $L^1$ density. Consider the cone
with apex $a$ and parametrize it by $\Phi(s,y)\coloneqq a+s(y-a)$, $(s,y)\in(0,1)\times D_\varepsilon$. Let $n$ be a unit normal of $D_\varepsilon$ and $e_1,\dots,e_{d-1}$ an orthonormal basis of its supporting
hyperplane. Then $\partial_s\Phi = y-a$ and $\partial_{e_i}\Phi = s\,e_i$, so $D\Phi$ has the columns
$y-a,\ se_1,\dots,se_{d-1}$ and
\begin{align*}
    J\Phi(s,y) \ &= \ \bigl|\det[\,y-a\mid se_1\mid\cdots\mid se_{d-1}\,]\bigr| \\
     &=  \ \bigl|\det[ se_1\mid\cdots\mid se_{d-1}\mid y-a ]\bigr| \ = \ s^{d-1}\,|(y-a)\cdot n| \ = \ \tfrac L2\,s^{d-1} \ > \ 0 ,
\end{align*}
having used that $\det(A)=\Pi_i A_{i,i}$ for a triangular matrix $A\in\R^{d\times d}$. Hence $\Phi$ is a
diffeomorphism onto the open cone minus its apex. {By direct computation
$\int_{D_\varepsilon}\sigma_{a,y}\,d\mu_\varepsilon(y) = \Phi_{\#}(\tau_y\,\rho)$ with
$\tau_y = (y-a)/|y-a|$ and $\rho\coloneqq|y-a|\,ds\otimes\mu_\varepsilon$ on $(0,1)\times D_\varepsilon$,
so its total variation is at most $\Phi_{\#}\rho$. In the coordinates
$(s,z)\in(0,1)\times\{|z|<\varepsilon\}$, $y = p+\sum_iz_ie_i$, one has
$\mu_\varepsilon = (\omega_{d-1}\varepsilon^{d-1})^{-1}\,dz$, hence
$\rho = \frac{|y-a|}{\omega_{d-1}\varepsilon^{d-1}}\,ds\,dz$, and the change of variables
$x = \Phi(s,y)$, $dx = J\Phi\,ds\,dz$, gives for every Borel set $E\subset\R^d$
\[
    \Phi_{\#}\rho(E) \ = \ \int_{\Phi^{-1}(E)}\frac{|y-a|}{\omega_{d-1}\varepsilon^{d-1}}\,ds\,dz
    \ = \ \int_E\frac{|y-a|}{\omega_{d-1}\varepsilon^{d-1}\,J\Phi(s,y)}\,dx ,
    \qquad (s,y) = \Phi^{-1}(x) .
\]
Thus $\Phi_{\#}\rho$ has, at $x = \Phi(s,y)$, the density
\[
    \frac{|y-a|}{\omega_{d-1}\varepsilon^{d-1}}\,\frac{1}{J\Phi(s,y)}
    \ = \ \frac{2\,|y-a|}{\omega_{d-1}\varepsilon^{d-1}L}\,s^{1-d}
    \ = \ \frac{2\,|y-a|^{d}}{\omega_{d-1}\varepsilon^{d-1}L}\,|x-a|^{1-d}
    \ \le\ C\,|x-a|^{1-d} ,
\]
because $s = |x-a|/|y-a|$ and $|y-a|\le\sqrt{L^2/4+\varepsilon^2}$; and $|x-a|^{1-d}$ is locally
integrable, $\int_{B_r(a)}|x-a|^{1-d}\,dx = d\,\omega_d\,r$.} The same applies at $b$.
\end{proof}
We are finally ready to prove our main technical filling.
\begin{proof}[Proof of Lemma~\ref{lem:filling}]
Set $f = \Div v$, so that $\fl f\le\int_\Omega|v|<\infty$ by \eqref{eq:apriori}. Apply
Proposition~\ref{prop:AE} with $\varepsilon = \eta/4$: there exist a compact $K\subset\Omega$, points
$a_i,b_i\in K$ and weights ${\lambda_i}>0$ with 
\[
    f \ = \ \sum_i{\lambda_i}\bigl(\delta_{a_i}-\delta_{b_i}\bigr),
    \qquad \sum_i{\lambda_i}\,d_\Omega(a_i,b_i) \ \le\ \fl f+\tfrac\eta4 ,
\]
the indices with $a_i = b_i$ being discarded. By Remark~\ref{rem:polygonal} we may choose, for each $i$, a polygonal path $\gamma_i \subset \Omega$ from $a_i$ to $b_i$ with 
\[
    \operatorname{length}(\gamma_i) \ \le\  d_\Omega(a_i,b_i) \  + \ 2^{-i}\frac{\eta}{4{\lambda_i}} ,
\]
and we write its segments as $\llbracket p^i_{j-1},p^i_j\rrbracket$, $j = 1,\dots,N_i$, with
$p^i_0 = a_i$ and $p^i_{N_i} = b_i$. Each $\gamma_i$ is a compact subset of $\Omega$, so
$\varrho_i\coloneqq\dist(\gamma_i,\partial \Omega)>0$. Thicken its segments: choose radii
$\varepsilon^i_j\in(0, \varrho_i)$ so small that, writing $L^i_j\coloneqq|p^i_j-p^i_{j-1}|$, we have
\[
    {\lambda_i}\sum_{j=1}^{N_i}\frac{2(\varepsilon^i_j)^2}{L^i_j} \ \le\ 2^{-i}\frac\eta4
    \qquad\text{and}\qquad
    \sum_{j=1}^{N_i}\frac{{\omega_{d-1}}}{d}\,(\varepsilon^i_j)^{d-1}L^i_j \ \le\ 2^{-i}\eta ,
\]
and set
\[
    {w} \ \coloneqq\ \sum_i{\lambda_i}\sum_{j=1}^{N_i}\sigma^{\varepsilon^i_j}_{p^i_{j-1},p^i_j},
\]
where $\sigma^{\varepsilon}_{a,b}$ is precisely given by Lemma \ref{lem:thicken}, invoked for each triple  $\varepsilon^i_j, a=p^i_{j-1} , b=p^i_j $.
Every summand is supported within distance $\varepsilon^i_j< \varrho_i$ of $\gamma_i$, hence $w$ is supported in $\Omega$. By Lemma~\ref{lem:thicken} the series converges absolutely in $L^1$, with
\begin{align*}
    \int_\Omega|{w}|
    \ &\le\ \sum_i{\lambda_i}\sum_{j=1}^{N_i}\Bigl(L^i_j+\frac{2(\varepsilon^i_j)^2}{L^i_j}\Bigr)
    \ \le\ \sum_i{\lambda_i}\operatorname{length}(\gamma_i)+\frac\eta4\\
    \ &\le\ \fl f+\frac\eta4+\frac\eta4 \ \le\ \fl f+\eta ,
\end{align*}
and the volume of $\{w\neq 0\}$ is at most $\sum_i2^{-i}\eta\le\eta$. Finally $\Div$ is continuous
along $L^1$-convergent series and the divergences telescope along each path, so
\[
    \Div {w} \ = \ \sum_i{\lambda_i}\sum_{j=1}^{N_i}\bigl(\delta_{p^i_{j-1}}-\delta_{p^i_j}\bigr)
    \ = \ \sum_i{\lambda_i}\bigl(\delta_{a_i}-\delta_{b_i}\bigr) \ = \ f
\]
in $\mathcal D'(\Omega)$.

Finally, assume in addition that $f$ is a measure supported on a compact negligible set
$K\subset\Omega$. Since $\fl f<\infty$, Lemma~\ref{lem:refine} applies with $\varepsilon=\eta/4$
and yields $a_i,b_i\in K$, $\lambda_i\in\R\setminus\{0\}$ and polygonal paths
$\gamma_i\subset\Omega$ from $a_i$ to $b_i$ with
\[
    f=\sum_i\lambda_i(\delta_{a_i}-\delta_{b_i}),\qquad
    \sum_i|\lambda_i|\,\ell(\gamma_i)\le\fl f+\tfrac\eta4,
\]
and with $\Sigma\coloneqq K\cup\bigcup_i\gamma_i$ compact and negligible. Interchanging $a_i$
with $b_i$ and reversing $\gamma_i$ we may assume $\lambda_i>0$. Since
$\vol(\{\dist(\cdot,\Sigma)\le s\})\downarrow\vol(\Sigma)=0$ as $s\downarrow0$, we may pick
$s\in(0,\dist(\Sigma,\Omega^c))$ such that the open set $U\coloneqq\{\dist(\cdot,\Sigma)<s\}$
satisfies $U\subset\Omega$ and $\vol(\overline U)\le\eta$. Repeating the construction
above with these paths, and imposing on the radii the additional constraint
$\varepsilon^i_j<\dist(\gamma_i,U^c)$, we obtain $w\in L^1(\Omega;\R^d)$ with $\Div w=f$ and
$\int_\Omega|w|\le\fl f+\eta$, all of whose summands are supported in $U$. Hence
$\spt w\subset\overline U$ and $\vol(\spt w)\le\eta$.
\end{proof}

\begin{proof}[Proof of Lemma~\ref{lem:modification}]
Let $A_n$ and $t(n)$ be as in Proposition~\ref{prop:truncation}. By Lemma~\ref{lem:bump} below, there are open sets $\Omega_i\Subset\Omega$ with smooth boundary
and $\vol(\Omega\setminus\Omega_i)\le i^{-1}$. Choose $i(n)\to\infty$ sufficiently fast so  that $\vol(\Omega\setminus\Omega_{i(n)})\le t(n)^{-1}$. Enlarging
the truncation set to $ A_n'\coloneqq A_n\cup(\Omega\setminus\Omega_{i(n)})$ we have
$\vol( A_n')\le(C+1)t(n)^{-1}\to0$, and
\[
    S_n\coloneqq V_n\mathds 1_{ A_n'},\qquad E_n\coloneqq V_n-S_n
\]
satisfy $\|E_n\|_{L^1}\le t(n)^{-1}$, with $E_n$ \emph{compactly supported} in $\Omega$. Since
$\Div V_n = 0$, the load shed by the truncation is
\[
    f_n \ \coloneqq\ \Div S_n \ = \ -\Div E_n ,
\]
an $m$-tuple of loads, one for each row, each of the form treated in Lemma~\ref{lem:filling}. 

We also observe that $f_n$ is a measure carried (and supported) by a compact negligible subset of
$\Omega$. Indeed, $B_n\coloneqq\Omega_{i(n)}\setminus A_n$ satisfies
$\overline{B_n}\subset\overline{\Omega_{i(n)}}\Subset\Omega$ and has finite perimeter,
being the difference of a set with smooth boundary and a set of locally finite perimeter
in $\Omega$ (Proposition~\ref{prop:truncation}). Since $V_n\in C^1$ and $\Div V_n=0$,
the Leibniz rule and De Giorgi's structure theorem (\cite{AmbrosioFuscoPallara00}) give
\[
    f_n=-\Div E_n=-\Div(V_n\mathds 1_{B_n})=-V_n\,D\mathds 1_{B_n}
    =(V_n\,\nu_{B_n})\,\Haus^{d-1}\restr\partial^*B_n ,
\]
where $\nu_{B_n}$ is the outer unit normal on the reduced boundary $\partial^*B_n$.
Hence each row of $f_n$ is a finite measure supported in
$K_n\coloneqq\partial B_n\subset\partial\Omega_{i(n)}\cup\partial_\Omega A_n$. Moreover $K_n$ is compact because $\partial B_n\subset\overline{\Omega_{i(n)}}$. It is also negligible, because
$\partial\Omega_{i(n)}$ is a compact smooth hypersurface and $\vol(\partial_\Omega A_n)=0$.

In particular each $f_n$ is a measure  satisfying the additional assumptions of the last conclusion in Lemma~\ref{lem:filling}. Applying Lemma~\ref{lem:filling} row by row with $\eta = t(n)^{-1}$, and using
\eqref{eq:apriori}, produces $P_n\in L^1(\Omega;\X)$ 
\[
    \Div P_n = f_n, \qquad \|P_n\|_{L^1}\le\sqrt m\,\|E_n\|_{L^1}+m\,t(n)^{-1}, \qquad
    \vol\bigl(\spt P_n\bigr)\le m\,t(n)^{-1} .
\]
Set $U_n\coloneqq S_n-P_n$. Then $\Div U_n = f_n-f_n = 0$ in $\mathcal D'(\Omega)$,
\[
    \|V_n-U_n\|_{L^1}\ \le\ \|E_n\|_{L^1}+\|P_n\|_{L^1}\ \le\ (1+\sqrt m+m)\,t(n)^{-1}\ \longrightarrow\ 0 ,
\]
and $\spt U_n\subset A_n'\cup \spt P_n$ has volume at most $(C+1+m)t(n)^{-1}\to0$. This settles the first conclusion with the Borel set $A_n'' \coloneqq A_n'\cup \spt P_n$.

Finally, we prove the last conclusion. The set $A_n''$ is relatively closed in $\Omega$:
$A_n$ is (Proposition~\ref{prop:truncation}), $\Omega\setminus\Omega_{i(n)}$ is, and
$\spt P_n$ is by definition. Hence $\Omega\setminus A_n'' \Subset \Omega$ is a bounded open
set, and Lemma~\ref{lem:bump} applied to $\Omega\setminus A_n''$ with $\eta=t(n)^{-1}$ provides an open set
$O_n\Subset \Omega\setminus A_n''$ with smooth boundary and $\vol(O_n)\ge\vol(\Omega\setminus A_n'')-t(n)^{-1}$. Consider the open subset 
\[
    \tilde A_n\coloneqq\Omega\setminus\overline{O_n} \qquad \Longrightarrow \qquad \partial \Omega \subset \partial\tilde A_n,\quad \partial_\Omega \tilde A_n = \partial O_n \Subset \Omega.
\]
This shows that $\partial_\Omega  \tilde A_n$ is a compact smooth hypersurface (cf.\ Lemma~\ref{lem:bump}). Moreover, 
\[
   \overline{O_n}\subset \Omega \setminus A_n'' \qquad \Longrightarrow \qquad \spt U_n \subset A_n'' \subset \tilde A_n.
\]
Because $\vol(\partial O_n)=0$, we further get
\begin{align*}
     \vol(\tilde A_n)=\vol(\Omega)-\vol(O_n)& \le\vol(\Omega)-\vol(\Omega \setminus A_n'')+t(n)^{-1} \\
     & =\vol(A_n'')+t(n)^{-1}\le(C+2+m)\,t(n)^{-1}\longrightarrow0 .
\end{align*}
Thus $\tilde A_n$ has all the required properties.
\end{proof}
For the following lemma, we refer e.g.\ \cite[Proposition 8.2.1]{Daners08}.
\begin{lemma}\label{lem:bump}
    Let $\Omega\subset \R^d$ be a bounded and open set. Then, for every $\eta>0$, there is $\delta>0$ and an open set $O_\delta$ with smooth boundary such that $O_\delta \Subset \Omega$ and
    \[
        \dist(O_\delta, \Omega^c)\ge \frac{\delta}{2} ,\qquad \vol(O_\delta)\ge \vol(\Omega) -\eta.
    \]
\end{lemma}

%=====================================================================
\section{What changes for symmetric fields}\label{sec:symmetric}
%=====================================================================
Bouchitté's conjecture was originally stated for \emph{stress} fields, $\Omega\to\R^{d\times
d}_{\mathrm{sym}}$ with $\Div\Sigma = 0$; the non-symmetric formulation used here is the geometric
one proposed by Alberti \cite{Alberti21}. The passage between them is not a formality: our argument
of \S\ref{sec:filling} does not survive.

The step with no symmetric analogue is the thickening of Lemma~\ref{lem:thicken}. The reason lies in what the divergence operator sees, and is related to the distinction between the boundary operator for currents (oriented surfaces) and the first variation operator on varifolds (non-oriented surfaces).

Without the symmetric condition, the divergence of a matrix field reduces to the study of each of its rows. By Smirnov's theorem, it thus suffices to study the divergence associated with an oriented Lipschitz curve $\gamma : [0,1] \to \R^d$ with positive unit tangent orientation $\tau : \mathrm{im} \, \gamma \to S^{d-1}$. The area formula and the fundamental theorem of calculus give
\begin{align*}
    -\Div (\tau \, \mathcal H^1 \llcorner \mathrm{im} \, \gamma)(\varphi) & = \int_{\mathrm{im} \, \gamma} \partial_{\tau} \varphi \, d\mathcal H^1  \\
    & = \int_0^1 \frac{d}{ds} \varphi(\gamma) \, ds = \varphi(b) - \varphi(a)
\end{align*}
where $a = \gamma(0)$ and $b = \gamma(1)$. In particular, the distributional divergence of $- \tau \, \mathcal H^1 \llcorner \mathrm{im} \, \gamma$ is precisely the dipole measure $\delta_b - \delta_a$. The curvature or the discontinuities of the orientation of such curves are not accounted for by the divergence operator. 

If we now consider the non-oriented curve $\tau \otimes \tau \, \mathcal H^1  \llcorner \mathrm{im} \, \gamma$, then the divergence acts on the tensor with $(i,j)$ entries $\tau^i\tau^j$. This results in the appearance of a second derivative of $\gamma$, which due to the Leibniz rule must account for a (generalized) curvature term ``$\kappa(\gamma) = \frac{d}{dt} \tau(\gamma)$'':
\begin{align*}
     -\Div (\tau \otimes \tau \, \mathcal H^1 \llcorner \mathrm{im} \, \gamma)(\psi) & = \int_{\mathrm{im} \, \gamma} \partial_{\tau} \psi \cdot \tau \, d\mathcal H^1  \\
     & = \int_0^1 \frac{d}{ds}[\psi(\gamma) \cdot \tau(\gamma)] \, ds  - \int_0^1 \psi(\gamma) \cdot \frac{d}{ds} \tau(\gamma) \, ds  \\
     & = [\psi(b)\cdot\tau(b) - \psi(a) \cdot\tau(a)]  - \kappa(\psi) .
\end{align*}
Thus  the divergence of the non-oriented curve $- (\tau \otimes \tau \, \mathcal H^1 \llcorner \mathrm{im} \, \gamma)$ consists of the endpoint atomic vector measures $\tau(b) \, \delta_b - \tau(a) \, \delta_a$ and the distributional curvature which for a sufficiently regular curve has the form
\[
    \kappa = \vec H \ \mathcal H^1 \llcorner \mathrm{im} \, \gamma  \ + \ \sum_{y \in \Delta_\gamma} [\tau^+ - \tau^-] \delta_y 
\]
where $\vec H$ is the approximate curvature vector of $\gamma$ and $\Delta_\gamma$ is the set of points $y$ where $\gamma$ has a corner and the tangent jumps from $\tau^-$ to $\tau^+$. Divergence-free symmetric fields carried by curves are
\emph{stationary} varifolds, and these are rigid: by Allard--Almgren \cite{AllardAlmgren76}, a
stationary one-dimensional varifold with density bounded away from zero is locally a finite union of
straight segments meeting at balanced junctions (the sum of their associated co-normals must equate to zero, a situation that cannot occur for injective curves). This means that shape matters for the symmetric case. Corners in particular possess concentrated
curvature. Rounding it does
not help either, as one then must pay approximate curvature. One might try to fatten a segment by an absolutely continuous superposition of
curves with the same endpoints, but as soon as a curve is not a bar the curvature term $\kappa$ appears breaking the PDE.

\medskip
\noindent\textbf{Acknowledgments.} All authors acknowledges support by the European Union (ERC, ConFine, 101078057) and the MIUR Excellence Department Project awarded to the Department of Mathematics, University of Pisa, CUP I57G22000700001. F.N. also acknowledges support by the INdAM-GNAMPA Project ``Analisi e Gamma-convergenza per alcuni funzionali non locali'' (CUP E53C25002010001\#).

%=====================================================================
\begingroup
\footnotesize
\setlength{\parskip}{0pt}

\endgroup

\medskip
\begingroup\footnotesize{
\noindent\textsc{(A.~Arroyo-Rabasa, I.~Y.~Violo) Dipartimento di Matematica, Università di Pisa,
Largo Bruno Pontecorvo 5, 56127 Pisa, Italy}\\
\noindent\textit{Email addresses}: \texttt{adolfo.rabasa@unipi.it}, \texttt{ivanyuri.violo@dm.unipi.it}

\medskip
\noindent\textsc{(F.~Nobili) Dipartimento di
Matematica e Applicazioni “Renato Caccioppoli”, Università di Napoli Federico II, Via
Cintia, Monte S.~Angelo, 80126 Napoli, Italy}\\
\noindent\textit{Email address}: \texttt{francesco.nobili@unina.it}}
\endgroup

\end{document}